\documentclass[12pt]{amsart}
\usepackage[utf8]{inputenc}
\usepackage[T1]{fontenc}
\usepackage{lmodern}
\usepackage{newtxtext,newtxmath}
\usepackage{amsxtra,amssymb,amsthm,amsmath,amscd,url,mathrsfs,mathtools}
\usepackage{tikz}
\usetikzlibrary{positioning,decorations.pathmorphing}
\usepackage{fullpage}
\usepackage{color} 
\usepackage[initials]{amsrefs}
\usepackage[colorlinks]{hyperref}

\newenvironment{tikzgraph}
  {\begin{tikzpicture}
      [vertex/.style={circle, draw=black, fill=black, inner sep=0.5pt, minimum
        size=6pt},
       edge/.style={thick},%
       medium/.style={line width=.75mm},%
       mediums/.style={line width=.5mm, decorate, decoration=snake},%
      ]\begin{scope}}
  {\end{scope}\end{tikzpicture}}

\newtheorem{fact}{Fact}

\newcommand{\sP}{\mathscr{P}}
\newcommand{\sE}{\mathscr{E}}

\theoremstyle{plain}
\newtheorem{theorem}{Theorem}[section]

\newtheorem{proposition}[theorem]{Proposition}

\newtheorem{conj}[theorem]{Conjecture}

\theoremstyle{definition}

\renewcommand{\le}{\leqslant}
\renewcommand{\ge}{\geqslant}

\newcommand{\abs}[1]{\left\lvert#1\right\rvert}

\newcommand{\eintervalle}[4]{\mathopen{#1}#2,\dotsc,#3\mathclose{#4}}
\newcommand{\icc}[2]{\eintervalle{\{}{#1}{#2}{\}}}

\title{Variations on the Petersen Colouring Conjecture}
\author{François Pirot}
\address{\'Equipe Orpailleur, LORIA (Université de Lorraine, C.N.R.S., INRIA),
Vandœuvre-lès-Nancy, France and
Department of Mathematics, Radboud University Nijmegen, Netherlands.}
\email{francois.pirot@loria.fr}

\author{Jean-S\'ebastien Sereni}
\address{Centre National de la Recherche Scientifique (ICube, CSTB), Strasbourg, France.}
\email{sereni@kam.mff.cuni.cz}

\author{Riste Škrekovski}
\address{Faculty of Information Studies, Novo Mesto,
Faculty of Mathematics and Physics, University of Ljubljana, and
FAMNIT, University of Primorska, Koper, Slovenia.}
\email{skrekovski@gmail.com}

\thanks{This work was partially supported by P.H.C. Proteus~[37455VB]; ARRS~[BI-FR-PROTEUS/17-18-009, P1-0383].}
\date{\today}

\begin{document}

\def\expanded{0}

\begin{abstract}
    The Petersen colouring conjecture states that every bridgeless cubic graph
    admits an edge-colouring with~$5$ colours such that for every edge~$e$, the
    set of colours assigned to the edges adjacent to~$e$ has cardinality
    either~$2$ or~$4$, but not~$3$.  We prove that every bridgeless cubic
    graph~$G$ admits an edge-colouring with~$4$ colours such that at
    most~$\frac45\cdot\abs{V(G)}$ edges do not satisfy the above condition.
    This bound is tight and the Petersen graph is the only connected graph for which
    the bound cannot be decreased.  We obtain such a $4$-edge-colouring by
    using a carefully chosen subset of edges of a perfect matching, and the
    analysis relies on a simple discharging procedure with essentially no
    reductions and very few rules.
\end{abstract}

\maketitle

\section{Introduction}
At the ninth British Combinatorial Conference in 1983, Fouquet and
Jolivet~\cite{FoJo83} introduced strong edge-colourings of cubic graphs. This
notion was further studied by Jaeger, who formulated a conjecture which is,
arguably, one of the most challenging conjecture in graph theory. Proving
Jaeger's conjecture to be true would have tremendous consequences, such as
confirming the Cycle double cover conjecture, the Berge-Fulkerson conjecture
and the nowhere-zero $5$-flow conjecture.

For any integer~$k\ge3$, consider
a $k$-edge-colouring of a cubic graph~$G=(V,E)$, that is, a mapping~$f\colon
E\to\icc{1}{k}$ such that $f(e)\neq f(e')$ for every two edges~$e$ and~$e'$
that share a vertex.  For an edge~$e\in E$, let~$\sE(e)$ be the set of four
edges adjacent to~$e$. The edge~$e$ is \emph{rich} if $\abs{f(\sE(e))}=4$,
while it is \emph{poor} if $\abs{f\left(\sE(e)\right)}=2$. The
edge-colouring~$f$ is \emph{normal} if every edge is either rich or poor.
The Petersen colouring conjecture reads as follows.
\begin{conj}[The Petersen coloring conjecture---Jaeger, 1985]
Every cubic bridgeless graph admits a normal $5$-edge-colouring.
\end{conj}

Now it is maybe a good time to explain the links with the ubiquitous Petersen
graph~$\sP$. A \emph{Petersen colouring} of a cubic graph~$G=(V,E)$ is a
mapping~$g$ that associates to each edge of~$G$ an edge of~$\sP$ such that if
two edges~$e$ and~$e'$ of~$G$ share a vertex, then so do the edges~$g(e)$
and~$g(e')$ of~$\sP$. As observed by Jaeger~\cite{Jae85}, normal colourings and
Petersen colourings of cubic graphs are in one-to-one correspondence.

Indeed, as is well known, the Petersen graph~$\sP$ can be seen as the Kneser
graph with parameters~$5$ and~$2$, defined as follows: the vertices are in
one-to-one correspondence with the $2$-element subsets of~$\icc{1}{5}$ and two
vertices are adjacent if and only if the corresponding subset are disjoint.
With this definition in mind, we can label every edge~$uv$ of~$\sP$ by the
unique integer~$\ell(uv)\in\icc{1}{5}$ that does not belong to~$X_u\cup X_v$,
where~$X_w$ is the $2$-element subset of~$\icc{1}{5}$ that corresponds
to~$w$, for every vertex~$w$ of~$\sP$.
Notice that if~$g$ is a Petersen colouring of a cubic graph~$G=(V,E)$, then
$f\colon E\to\icc{1}{5}$ defined by~$f(e)\coloneqq\ell(g(e))$ is a normal
colouring of~$G$.

Conversely, assume that~$f$ is a normal $5$-edge-colouring of a cubic
graph~$G=(V,E)$. Keeping in mind the labelling~$\ell$ of the edges of~$\sP$ given above,
we define the mapping~$g\colon E\to E(\sP)$ as follows.
For each edge~$e=uv\in E$, we define~$g(e)$ to be the edge~$e'$ of~$\sP$ such that
first~$\ell(e')=f(e)$, and second~$e'$ is incident to the vertex~$w\in V(\sP)$ such that
the three colours assigned by~$f$ to the edges of~$G$ incident to~$u$ are the elements
of~$X_w\cup\{f(e)\}$.
A straightforward checking ensures that~$g$ is a Petersen colouring of~$G$.  We
just proved the following equivalence, which was first established by
Jaeger~\cite{Jae85}.
\begin{proposition}
  Let~$G$ be a cubic graph. Then,~$G$ admits a normal colouring if and only
  if~$G$ admits a Petersen colouring.
\end{proposition}

Notice that a $3$-edge-colouring of a (connected) cubic graph~$G$ is precisely a
normal colouring in which every edge is poor.  Fouquet and
Jolivet~\cite{FoJo83} coined the term \emph{strong colouring} to define an
edge-colouring in which every edge is rich. This corresponds precisely to an
edge-colouring ``at distance~$2$'' or, in other words, to a vertex-colouring of
the square of the line graph of~$G$.  The unique normal colouring of the
Petersen graph is a strong colouring.

Despite original approaches~\cite{MaSk05,NeSa08,Sam11}, few progress has been made on the
Petersen colouring conjecture: ways to infirm it remain elusive as possible
counter-examples must be snarks, that is bridgeless cubic graphs that are not
$3$-edge-colourable (the ones we know are usually obtained from well-structured
graph operations, for which the Petersen colouring conjecture can be verified~\cite{Bil15,HaSt14}),
and confirming the conjecture is expected to be a difficult task since as
reported earlier this would confirm several difficult and most researched graph
conjectures.

In view of the difficulty of the question, it is natural to ask for weaker versions
of the conjecture. Because a strong colouring is normal, we know that every cubic
graph admits a normal colouring using at most~$10$ colours: indeed Andersen~\cite{And92}
and, independently, Horák, He and Totter~\cite{HHT93} established this statement (which
confirms, for the particular case of graphs with maximum degree~$3$,
a conjecture of Erd\H{o}s and Nešetřil formulated in~1985 during a seminar in Prague).
Further, it has been noted before~\cite{Bil15} that Seymour's $8$-flow theorem provides,
for any bridgeless cubic graph, a normal $7$-edge-colouring. To the best of our knowledge,
whether a normal $6$-edge-colouring can be found for any such graph is still an open question.
A line of study is then to find an edge-colouring that is normal on a large proportion of
the graph, which we formalise in the next subsection. Before that, we end this part
with a remark.

As mentioned earlier, a $3$-edge-colourable graph~$G$ always has a Petersen
colouring: let~$e_1$, $e_2$ and~$e_3$ be the three edges incident to an
arbitrary vertex~$v$ of the Petersen graph. Label~$e_i$ by~$\ell(e_i)\coloneqq
i$ for each~$i\in\{1,2,3\}$.  If~$c$ is a $3$-edge-colouring of~$G$, then
defining~$f\colon E(G)\to E(\sP)$ by~$f(e')\coloneqq e_{c(e')}$ yields a
Petersen colouring of~$G$. This Petersen colouring is, in some sense, trivial:
it only uses the incidences at a single vertex of the Petersen graph.  This is
equivalent to saying that the original graph admits a $3$-edge-colouring.  A
fact that seems worth pointing out, however, is that Petersen colourings of
bridgeless cubic graphs are either ``trivial'' (and hence the graph admits a
$3$-edge-colouring) or surjective.

\subsection{The rich, the poor and the medium.}
Given an edge-colouring of a cubic graph~$G$, define an edge~$e$ to be \emph{medium}
if it is neither rich nor poor. Since the Petersen colouring conjecture states that
every bridgeless cubic graphs admits a $5$-edge-colouring such that no edge is medium,
it seems interesting to investigate the minimum number of medium edges in edge-colourings
of bridgeless cubic graphs. As observed by Bílková~\cite{Bil15}*{p.~9}, Petersen's perfect matching theorem
combined with Vizing's edge-colouring theorem (and some further analysis if the
graph has cycles of length less than~$5$) directly yield for every bridgeless
cubic graph a $5$-edge-colouring such that at least one third of the edges are rich or poor.
This lower bound was improved~\cite{Bil15}*{Theorem~3.2} to two thirds of the edges for cubic graphs having a $2$-factor
consisting of two cycles of the same length --- the class of ``generalised prisms'' ---
and to roughly half the edges in graphs with no short cycles~\cite{Bil15}*{Theorem~3.6}.

Ways how poor, rich and medium edges combine in edge-colourings looks
intriguing. We already pointed out that an edge-colouring with poor edges only
is actually a $3$-edge-colouring. Oppositely, an edge-colouring with rich edges
only is a strong colouring. Notice that every normal $5$-edge-colouring of the Petersen
graph actually contains no poor edge: it is thus a strong colouring.

We consider $4$-edge-colourings and prove the following result.

\begin{theorem}\label{thm-main}
Every connected cubic bridgeless graph~$G$ admits a $4$-edge-colouring such that
    at most~$\frac45\cdot\abs{V(G)}$ edges are neither rich nor poor.
Furthermore, no~$4$-edge-colouring of~$G$ yields less medium edges if and only
if~$G$ is the Petersen graph.
\end{theorem}

Since an $n$-vertex cubic graph has~$\frac{3n}{2}$ edges,
Theorem~\ref{thm-main} ensures that every bridgeless cubic graph~$G$ admits a
$4$-edge-colouring containing at most~$\frac{8}{15}\cdot\abs{E(G)}$ medium
edges. This bound cannot be improved in general, since the Petersen graph
has~$15$ edges and each of its $4$-edge-colourings yields at least~$8$
medium edges.

\begin{figure}[!th]
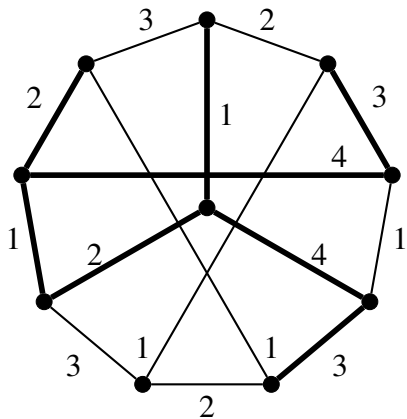
 
    \begin{center}
\begin{tikzgraph}
 \foreach \i in {0,1,...,8}
    \draw (90+40*\i:2.5cm) node[vertex] (a\i) {};
    \draw (0,0) node[vertex] (a10) {};
    \draw[edge] (a0)--(a1) node[above,midway] {$3$} --(a2) node[above left,midway] {$2$} --(a3) node[left,midway] {$1$} --(a4) node[below left,midway] {$3$} --(a5) node[below,midway] {$2$} --(a6) node[below right,midway] {$3$} --(a7) node[right, midway] {$1$} --(a8) node[above right,midway] {$3$} --(a0) node[above,midway] {$2$};
    \draw[edge] (a3)--(a10) node[left=1.5mm,midway] {$2$} --(a0) node[right,midway] {$1$};
    \draw[edge] (a6)--(a10) node[right=1.5mm,midway] {$4$};
    \draw[edge] (a8)--(a4) node[above=2mm] {$1$};
    \draw[edge] (a1)--(a5) node[above=2mm] {$1$};
    \draw[edge] (a2)--(a7) node[xshift=-7mm,yshift=2.5mm] {$4$};
\draw[medium] (a1)--(a2)--(a3)--(a10)--(a0);
\draw[medium] (a8)--(a7);
\draw[medium] (a2)--(a7);
\draw[medium] (a10)--(a6);
\draw[medium] (a5)--(a6);
\end{tikzgraph}
\end{center}
    \caption{A $4$-edge-colouring of the Petersen graph with exactly~$8$ medium edges, drawn thicker.}
\end{figure}

\section{Proof of \texorpdfstring{Theorem~\ref{thm-main}}{Theorem 1.3}}

We demonstrate the upper bound: every bridgeless cubic graph~$G$
admits a $4$-edge-colouring such that at most~$\frac45\cdot\abs{V(G)}$
edges are medium. While developing the proof, we shall see that the only
case where the bound must be attained is if~$G$ is the Petersen graph.
We proceed by induction on the number of vertices of~$G$.
The induction yields the conclusion in a standard way if~$G$ has a triangle, and
we thus first deal with this case.  To make the argument smoother, we actually
prove the result for (loopless) bridgeless cubic multi-graphs.
For instance any $4$-edge-colouring of a triple edge between two vertices
yields three poor edges. Similarly, if a cubic graph~$G$ contains two vertices
with exactly two parallel edges between them, then in any $4$-edge-colouring
of~$G$ either both edges are poor or both edges are medium.

As reported earlier, every $3$-edge-colouring of a cubic graph contains only poor edges,
and hence the statement of Theorem~\ref{thm-main} is correct if~$G$
admits a $3$-edge-colouring, and hence in particular if~$\abs{V(G)}=2$.
We hence consider a connected bridgeless cubic multi-graph~$G$
that admits no $3$-edge-colouring, and we set~$n\coloneqq\abs{V(G)}$
(so~$n\ge10$). Our first two arguments are standard and well known to people used
to graph colouring but they are included for completeness.

We use induction to prove the statement if~$G$ contains a multi-edge.
Indeed, suppose that~$e_1$ and~$e_2$ are two different edges with end-vertices~$v_1$ and~$v_2$.
For each~$i\in\{1,2\}$, let~$u_i$ be the neighbour of~$v_i$ different from~$v_{3-i}$.
Since~$n>2$ and because~$G$ is bridgeless,~$u_1\neq u_2$. Let~$G'$ be the bridgeless
cubic multi-graph obtained from~$G-\{v_1,v_2\}$ by adding a new edge~$e'$ between~$u_1$ and~$u_2$.
The induction hypothesis ensures that~$G'$ admits a $4$-edge-colouring~$c'$ yielding
at most~$\frac45\cdot(n-2)$ medium edges. It is straightforward to deduce from~$c'$ a $4$-edge-colouring~$c$
of~$G$ with no more medium edges, which thus prove the statement (including ``the furthermore part'')
if~$G$ contains a multi-edge. 
\if\expanded1
Without loss of generality, we may assume that~$c'(e')=1$, the other
two edges incident to~$u_1$ are coloured~$2$ and~$3$, and one of the other edges incident to~$u_2$
is coloured~$2$. There are two possibilities for the third edge incident to~$u_2$: it is coloured
either~$3$ or~$4$ (see Figure~\ref{fig-multi}.)

\begin{figure}[!th] 
\begin{tikzgraph}
\node[vertex] (v1) at (5,1) {};
\draw[above] (v1) node {$v_1$};
\node[vertex] (v2) at (7,1) {};
\draw[above] (v2) node {$v_2$};
\draw[thick,bend right]  (v1) edge (v2);
\draw[thick,bend left]  (v1) edge (v2);
\node[vertex] (u1) at (5,-.5) {};
\node[vertex] (u2) at (7,-.5) {};
\draw[right] (u1) node {$u_1$};
\draw[left] (u2) node {$u_2$};
\draw[thick] (v1)--(u1)--++(-.5,-.75);
\draw[thick] (v1)--(u1)--++(.5,-.75);
\draw[thick] (v2)--(u2)--++(-.5,-.75);
\draw[thick] (v2)--(u2)--++(.5,-.75);
    \path (v1)--(v2) node[above=4mm,midway] {$2$};
    \path (v1)--(v2) node[below=4mm,midway] {$3$};
    \path (v1)--(u1) node[left,midway] {$1$};
    \path (v2)--(u2) node[right,midway] {$1$};
    \path (u1)--++(-.3,-.35) node[left] {$3$};
    \path (u1)--++(.3,-.35) node[right] {$2$};
    \path (u2)--++(-.3,-.35) node[left] {$2$};
    \path (u2)--++(.3,-.35) node[right] {$3$ or~$4$};
\draw[very thick,<->] (1.5,-.5)--(2.5,-.5);
\node[vertex] (uu1) at (-3,-.5) {};
\node[vertex] (uu2) at (-1,-.5) {};
    \draw (uu1) node[above] {$u_1$};
    \draw (uu2) node[above] {$u_2$};
\draw[thick] (uu1)--(uu2);
\draw[thick] (uu1)--++(-.5,-.75);
\draw[thick] (uu1)--++(.5,-.75);
\draw[thick] (uu2)--++(-.5,-.75);
\draw[thick] (uu2)--++(.5,-.75);
    \path (uu1)--(uu2) node[above,midway] {$1$};
    \path (uu1)--++(-.3,-.35) node[left] {$3$};
    \path (uu1)--++(.3,-.35) node[right] {$2$};
    \path (uu2)--++(-.3,-.35) node[left] {$2$};
    \path (uu2)--++(.3,-.35) node[right] {$3$ or~$4$};
\end{tikzgraph}
    \caption{Standard reduction of multi-edges and of triangles: on the left
picture, the edge coloured~$2$ incident to~$u_1$ could be the same as that
coloured~$2$ incident to~$u_2$.}\label{fig-multi}
\end{figure}

Let~$c\colon E(G)\to\icc{1}{4}$ be defined by~$c(e)\coloneqq c'(e)$ if~$e\in E(G)\cap E(G')$,
$c(v_1u_1)\coloneqq c(v_2u_2)\coloneqq 1$, $c(e_1)\coloneqq 2$ and~$c(e_2)\coloneqq 3$.
Then~$c$ is a $4$-edge-colouring of~$G$ that yields no more medium edges than~$c'$ in~$G'$.
\else
One obtains a $4$-edge-colouring~$c$ of~$G$ by setting~$c(e)=c'(e)$ if~$e\in E(G)\cap E(G')$
and $c(v_iu_i)\coloneqq c'(e')$ for each~$i\in\{1,2\}$,
and letting~$c(e_1)$ and~$c(e_2)$ be the two colours of the two edges of~$G'$ incident to~$u_1$ that are
different from~$e'$. The obtained $4$-edge-colouring~$c$ of~$G$ yields at most~$\frac45\cdot(n-2)$ medium
edges, which is less than~$\frac45\cdot n$.
\fi
\noindent
Consequently, we may assume that~$G$ is simple.

We now use induction to prove the statement if~$G$ contains a triangle~$v_0v_1v_2$. Indeed,
we then define~$G'$ to be the multi-graph obtained from~$G$ by contracting~$v_0$, $v_1$ and~$v_2$
into a single vertex~$x$. It follows that~$G'$ is a bridgeless and cubic multi-graph with~$n-2$ vertices,
and therefore the induction hypothesis yields that~$G'$ admits a $4$-edge-colouring~$c'$ yielding at
most~$\frac45\cdot(n-2)$ medium edges.
\if\expanded1
This colouring~$c'$ can be extended into a $4$-edge-colouring~$c$
of~$G$ as follows. Note that the three edges of~$G'$ incident to~$x$ must have different colours under~$c'$.
First set~$c(e)\coloneqq c'(e)$ if~$e\in E(G)\cap E(G')$. Second for each~$i\in\{0,1,2\}$
set~$c(v_iv_{i+1})\coloneqq c(v_iu_i)$ where~$u_i$ is the neighbour of~$v_i$ not in~$\{v_{i+1},v_{i+2}\}$
and the indices are modulo~$3$. (See Figure~\ref{fig-triangle}.)
The $4$-edge-colouring~$c$ of~$G$ has no more medium edges than~$c'$
has in~$G'$.

\begin{figure}[!th] 
    \begin{center}
\begin{tikzgraph}
\begin{scope}[xshift=-5cm]
\draw (0,0) node[vertex] (x) {};
\draw[edge] (x)--++(0,.75cm);
\draw[edge] (x)--++(-.5,-.65);
\draw[edge] (x)--++(.5,-.65);
\path (x)--++(0,.6) node[left] {$1$};
\path (x)--++(-.3,-.4) node[left] {$2$};
\path (x)--++(.3,-.4) node[right] {$3$};
\draw[very thick,<->] (1.5,0)--(2.5,0);
\end{scope}
\begin{scope}[rotate=90]
 \foreach \x/\y in {0/0,120/1,240/2}{
 \node[vertex] (\y) at (canvas polar cs: radius=1cm,angle=\x){};
 }
\end{scope}
\draw[edge] (0)--(1) node[midway,left] {$3$} --(2) node[midway,below] {$1$} --(0) node[midway,right] {$2$} --++(0,.75cm) node [right,midway] {$1$};
\draw[edge] (1)--++(-.5,-.65) node[midway,left] {$2$};
\draw[edge] (2)--++(.5,-.65) node[midway,right] {$3$};
\end{tikzgraph}
    \end{center}
    \caption{The induction step if~$G$ contains a triangle.}\label{fig-triangle}
\end{figure}

\else
This colouring~$c'$ can be extended into a $4$-edge-colouring~$c$ of~$G$ by
setting~$c(e)\coloneqq c'(e)$ if~$e\in E(G)\cap E(G')$ and
$c(v_iv_{i+1})\coloneqq c'(v_{i+2}u_{i+2})$ where~$i\in\{0,1,2\}$ is considered
modulo~$3$ and~$u_i$ is the neighbour of~$v_i$ not in~$\{v_{i+1},v_{i+2}\}$.
The obtained $4$-edge-colouring~$c$ yields no more medium edges in~$G$ than~$c'$
does in~$G'$.
\fi
\noindent
Consequently, we may assume that~$G$ is a simple bridgeless cubic
graph with no triangle.

It remains to deal with the case where~$G$ is a bridgeless cubic graph with no triangles.
Let~$F$ a $2$-factor of~$G$ and~$M$ the perfect matching
such that~$F=G-M$.
Because~$G$ admits no $3$-edge-colouring, we know that~$F$ contains at least two odd cycles.
In particular, there exists an edge in~$M$ that is not a chord of a cycle in~$F$.
Note also that cycles of~$G$ with length (at most)~$5$ have no chord, as~$G$ has no triangle.

If~$v\in V(G)$, we define~$C_v$ to be the cycle in~$F$ to which~$v$ belongs and~$v'$ to be
the unique neighbour of~$v$ in~$G$ such that~$vv'\in M$.
For every cycle~$C\in F$, there is a cyclic ordering~$\varphi_C$ of the
vertices in~$V(C)$, which we extend to the edges in~$M$ that are
incident to a vertex in~$C$: two edges~$e$ and~$e'$ in~$M$ that have each
exactly one end-vertex on a given cycle~$C\in F$ are \emph{consecutive} if
their end-vertices in~$C$ are consecutive with respect
to~$\varphi_C$. (Notice that chords are purposely excluded from this definition.)

An \emph{edge-selection} is a subset~$S$ of~$M$ with the following properties:
\begin{enumerate}
    \item every edge in~$S$ is incident to two different odd cycles in~$F$ (in particular,
        no edge in~$S$ is a chord of a cycle in~$F$);
    \item every cycle in~$F$ is incident to at most two edges in~$S$; and
    \item if a cycle~$C\in F$ is incident to two edges in~$S$, then these two edges are consecutive.
\end{enumerate}
Given an edge-selection~$S$ and a cycle~$C\in F$, the \emph{degree~$\deg_S(C)$ of~$C$ in~$S$}
is the number of edges incident to~$C$ that belong to~$S$, and hence~$\deg_S(C)\in\{0,1,2\}$.
If in addition~$C'\in F$, then~$C$ and~$C'$ are \emph{$S$-adjacent} if~$S$
contains an edge incident to both~$C$ and~$C'$. An \emph{$S$-component} of~$G$
is an inclusion-wise maximal subset~$K$ of~$F$ such that for every two distinct cycles~$C$ and~$C'$ in~$K$, there exists
a sequence~$(C_i)_{0\le i\le t}$ of cycles in~$K$ such that~$C_0=C$, $C_t=C'$ and for every~$i\in\{0,\dotsc,t-1\}$,
the cycles~$C_i$ and~$C_{i+1}$ are $S$-adjacent. 
The edges in~$S$ joining two cycles in~$K$
are said to be \emph{associated with~$K$}. This relation is a surjective mapping
from~$S$ to the set of all $S$-components of~$G$.
We need a last definition.  If~$K$ is an $S$-component, then we let~$G_K$ be the
multi-graph with vertex set~$K$ and~$x$ edges between~$C$ and~$C'$ where~$x\in\{0,1,2\}$
is the number of edges in~$S$ that are incident to both~$C$ and~$C'$ in~$G$.
It follows from the definitions that~$G_K$ is either a single vertex, or a
path, or a cycle, or two vertices joined by two parallel edges.

Among all edge-selections of maximum order, we choose one such that the number
of $S$-degree-$2$ cycles is as large as possible.
We construct a $4$-edge-colouring of~$G$ with the following properties:
\begin{itemize}
    \item every edge in~$M$ is coloured~$4$ and no other edge is coloured~$4$;
    \item an edge is coloured~$3$ only if it belongs to an odd cycle in~$F$;
    \item every odd cycle in~$F$ has exactly one edge coloured~$3$;
    \item every edge in~$S$ is adjacent to two edges coloured~$3$; and
    \item if an edge~$e$ in~$S$ is medium, then it is associated with an
        $S$-component~$K$ such that~$G_K$ is an odd cycle and~$e$ is the only
        medium edge associated with~$K$.
\end{itemize}
To see why such a $4$-edge-colouring exists,
start by colouring the edges of~$G$ that belong to~$M$ with~$4$. Next
colour the edges of every even cycle in~$F$ using~$\{1,2\}$. By~(2)
and~(3) every odd
cycle~$C\in F$ has an edge that is incident to all the edges in~$S$ incident
to~$C$: colour this edge with~$3$.  The remaining uncoloured edges span a
vertex-disjoint collection of paths, and we colour them using~$\{1,2\}$,
independently for each $S$-component~$K$. If~$G_K$ is not an odd cycle, then
one can ensure that no edge associated with~$K$ is medium. If~$G_K$ is
an odd cycle, then one can ensure that exactly one edge~$S$ associated
with~$K$ is medium.

Our goal is to demonstrate that the obtained $4$-edge-colouring of~$G$ contains at most~$4n/5$ medium edges,
where~$n$ is the number of vertices of~$G$.
We use a discharging argument to count the medium edges: we start by assigning a charge of~$1$ to each medium edge,
and thus throughout all the process the total charge in the graph~$G$ is precisely the number of medium edges.

We shall define a number of discharging rules: in the first ones, medium edges send charge to cycles in~$F$
to which they are incident. Later, some cycles in~$F$ will send some charge to other cycles in~$F$.
We apply the rules in order and analyse the global state of the charge in the
graph after one or more rules have been applied. At the end, we prove that for
each $S$-component~$K$ the sum of the charges of the cycles in~$K$
is at most~$\frac45$ times the number of vertices belonging to cycles in~$K$, which
implies the sought upper bound on the number of medium edges.

\begin{fact}\label{as-1}
    A cycle in~$F$ contains zero medium edge if it is even and~$3$ medium edges if it is odd.
\end{fact}
\begin{proof}
Indeed, if the edge~$e$ belongs to an even cycle, then its colour~$c$ belongs
    to~$\{1,2\}$, and each of its end-vertices is incident to an edge
    coloured~$4$, which belongs to~$M$, and an edge coloured~$3-c$, which
    belongs to the even cycle. So~$e$ is poor.  Let~$C=v_0\dotso v_{2k}$ is an
    odd cycle in~$F$ with~$v_0v_1$ being its only edge coloured~$3$. For
    each~$i\in\{2,\dotsc,2k-1\}$, let~$c_i\in\{1,2\}$ be the colour of the
    edge~$v_iv_{i+1}$. Then the edge~$v_iv_{i+1}$ is incident to two edges
    coloured~$4$ and two edges coloured~$3-c_i$, and hence~$v_iv_{i+1}$ is
    poor. Consequently, the only medium edges on~$C$ are~$v_0v_1$,
    $v_1v_2$ and~$v_{2k}v_0$.
\end{proof}

\noindent
Our first rule reads as follows.

\bigskip
\noindent
\textbf{(R0)} Every medium edge that belongs to a cycle~$C$ in~$F$ sends~$1$ to~$C$.

\bigskip
After applying rule~(R0), an edge has charge~$0$ except if it is a medium edge that belongs to~$M$,
in which case it has charge~$1$. In addition, a cycle in~$F$ has charge~$0$ if it is even
and charge~$3$ if it is odd. For our next rule, notice that if~$e$ is a medium edge that belongs to~$M$,
then~$e$ is adjacent to an edge coloured~$3$, which must belong to an odd cycle in~$F$.

\begin{figure}[!t]
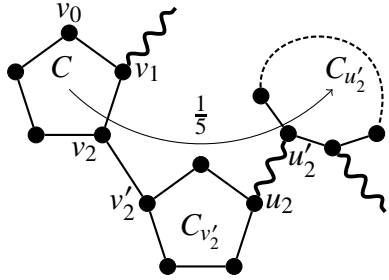
 
    \begin{center}
\begin{tikzgraph}
 \foreach \i in {0,1,2,3,4}{
    \draw (90+72*\i:.75cm) node[vertex] (a\i) {};
    \draw (90+72*\i:.75cm) node[vertex,xshift=1.75cm,yshift=-1.75cm] (b\i) {};
}
\foreach \i in {1,2,3,4} {
\draw (162+36*\i:1cm) node[vertex,xshift=3.5cm,yshift=.22cm] (c\i) {};
}
\foreach \i\j in {0/1,1/2,2/3,3/4,4/0} {
  \draw[edge] (a\i)--(a\j);
  \draw[edge] (b\i)--(b\j);
}
\foreach \i\j in {1/2,2/3,3/4} {
  \draw[edge] (c\i)--(c\j);
}
\draw[edge,dashed,dash pattern=on 2pt off 1pt] (c1)  ..controls (2.5,1) and (4.5,1).. (c4);
\draw[edge] (a3)--(b1);
\draw[mediums] (b4)--(c2);
\draw[mediums] (a4)--++(.7cm,.85cm);
\draw[mediums] (c3)--++(.7cm,-.85cm);
\node[above] at (a0) {$v_0$};
\node[right] at (a4) {$v_1$};
\node[below left=-.5mm and -.9mm] at (a3) {$v_2$};
\node[left] at (b1) {$v'_2$};
\node[right] at (b4) {$u_2$};
\node[below right=-.2mm and -1.5mm] at (c2) {$u'_2$};
\draw[->] (0,0) to[out=-45,in=-135] (3.5cm,0);
\node[above] at (1.75cm,-.75cm) {$\frac15$};
    \node at (-1mm,3mm) {$C$};
    \node[right] at (3.25cm,2mm) {$C_{u'_2}$};
    \path (b1)--(b4) node[below,midway] {$C_{v'_2}$};
\end{tikzgraph}
        \caption{Illustration of the discharging rule~(R4): $C$ and~$C_{v'_2}$
        both are cycles of length~$5$ and $S$-degree~$1$ while~$C_{u'_2}$ has
        $S$-degree~$2$; waved edges are those in~$S$ and one could
        have~$C_{v'_1}=C_{u'_2}$, that is, the two waved edges with only one
        end-vertex drawn could be the same edge.}\label{fig-R4}
    \end{center}
\end{figure}

\bigskip
\noindent
\textbf{(R1)} Let~$e$ be a medium edge that belongs to~$M$, and let~$C$ and~$C'$ be the two cycles in~$F$
to which~$e$ is incident, such that~$C$ is odd and the edge
coloured~$3$ on~$C$ is adjacent to~$e$. If~$C'$ is even, then~$e$
sends~$1/2$ to~$C'$ and~$1/2$ to~$C$. If~$C'$ is odd
its edge coloured~$3$ is adjacent to~$e$, then~$e$
sends~$1/2$ to each of~$C$ and~$C'$. If~$C'$ is odd and its edge coloured~$3$ is not adjacent to~$e$,
then~$e$ sends~$1$ to~$C$ and nothing to~$C'$.

\bigskip
After applying rule~(R1), every edge has charge~$0$. If~$C\in F$ is an even cycle,
then its charge is at most~$1/2\cdot|V(C)|$, which is less than~$4|V(C)|/5$.
In addition, if~$C\in F$ is an odd cycle, then one of the following occurs:
\begin{itemize}
    \item $C$ has~$S$-degree~$0$ and charge at most~$5$;
    \item $C$ has~$S$-degree~$1$ and charge at most~$4$; or
    \item $C$ has~$S$-degree~$2$ and charge at most~$7/2$.
\end{itemize}
Indeed, if~$C$ is incident to exactly one edge in~$S$, then this edge cannot be medium by the construction of
the $4$-edge-colouring. If~$C$ is incident to exactly two edges in~$S$, then at most one of them is medium,
in which case it sends~$1/2$ to~$C$.
It follows that the charge of~$C$ is at most than~$\frac45\cdot|V(C)|$ unless~$C$
has length~$5$ and $S$-degree~$0$.

The next step is to apply the
three following rules, the third one being illustrated in Figure~\ref{fig-R4}.

\bigskip
\noindent
\textbf{(R2)} If~$C$ is a cycle of length~$5$ in~$F$ with $S$-degree~$0$,
then~$C$ sends~$1/5$ to~$C_{v'}$ for each~$v\in V(C)$.

\bigskip
\noindent
\textbf{(R3)} Let~$C=v_0\dotso v_4$ be a cycle of length~$5$ in~$F$ of $S$-degree~$1$, with~$v_1v'_1$ being the unique edge in~$S$ incident to~$C$.
For each~$i\in\{0,2\}$, if~$C_{v'_i}$ is not a cycle of length~$5$ with $S$-degree~$1$,
then~$C$ sends~$1/5$ to~$C_{v'_i}$ through~$v_i$.

\bigskip
\noindent
\textbf{(R4)} Let~$C=v_0\dotso v_4$ be a cycle of length~$5$ in~$F$ of $S$-degree~$1$, with~$v_1v'_1$ being the unique edge in~$S$ incident to~$C$.
Let~$i\in\{0,2\}$. Suppose that~$C_{v'_i}$ is a cycle of length~$5$ of $S$-degree~$1$, written~$v'_iu_1u_2u_3u_4$ such that~$u_2$
is its unique vertex incident to an edge in~$S$.
If~$C_{u'_2}$ is a cycle of $S$-degree~$2$ then~$C$ sends~$1/5$ to~$C_{u'_2}$ through~$v_i$.

\bigskip
We now check that after applying~(R2)--(R4), for every $S$-component~$K$, the
sum of the charges of the cycles in the component is at most~$4/5$ times the
number of vertices belonging to cycles in~$K$ with equality only if~$K$
contains only cycles of length~$5$.
We first analyse the current charge of every cycle~$C\in F$.
If~$C$ is even, then it has charge at
most~$(1/2+1/5)\cdot\abs{V(C)}<\frac45\cdot\abs{V(C)}$.

If~$C$ is a cycle with $S$-degree~$0$ and length~$2k+1$, where~$k\ge2$, 
then~$C$ does not receive any charge.  Indeed, there cannot be an edge~$e$
in~$M$ incident to~$C$ and a cycle of length~$5$ and~$S$-degree~$0$,
as~$S\cup\{e\}$ would then contradict the maximality of~$S$. This shows
that~$C$ does not receive charge by~(R2). Moreover, if~$C$ would receive charge
by~(R3) then the definition of~(R3) would imply the existence of a
cycle~$C'=v_0v_1v_2v_3v_4$ in~$F$ of length~$5$ and~$S$-degree one such
that~$v_1v'_1\in S$ and~$C=C_{v'_0}$ (or~$C=C_{v'_2}$).
Consequently,~$S\cup\{v_0v'_0\}$ would contradict the maximality of~$S$.
Therefore, the final charge of~$C$ is at most~$5$ if~$k\ge3$, which is less
than~$4\cdot(2k+1)/5$. If~$k=2$ then by rule~(R2) the final charge of~$C$ is at most
$5-5\cdot\frac15=4=4\cdot(2k+1)/5$.

Let~$C$ be a cycle with~$S$-degree~$1$. 
Note that~$C$ can receive charge only because of~(R2) and~(R3). Moreover,~$C$
does not receive charge through its vertex incident to an edge in~$S$.
If~$C$ has length~$2k+1$ with~$k\ge3$, then
before applying~(R2)--(R4) the charge of~$C$ was at most~$4$ since the edge in~$S$ incident to~$C$
is not medium, and hence its final charge is at most $4+2k/5$, which is less
than~$4\cdot(2k+1)/5$ since~$k\ge3$.

\begin{figure}[!th]
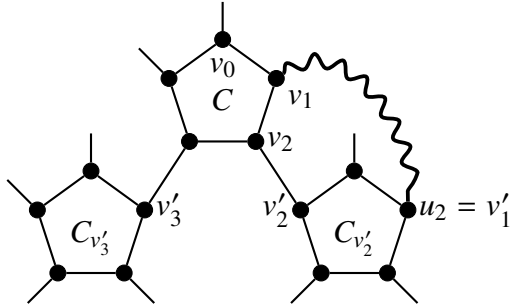
 
    \begin{center}
\begin{tikzgraph}
 \foreach \i in {0,1,2,3,4}{
    \draw (90+72*\i:.75cm) node[vertex] (a\i) {};
    \draw (90+72*\i:.75cm) node[vertex,xshift=1.75cm,yshift=-1.75cm] (b\i) {};
    \draw (90+72*\i:.75cm) node[vertex,xshift=-1.75cm,yshift=-1.75cm] (c\i) {};
}
\foreach \i\j in {0/1,1/2,2/3,3/4,4/0} {
  \draw[edge] (a\i)--(a\j);
  \draw[edge] (b\i)--(b\j);
  \draw[edge] (c\i)--(c\j);
}
\draw[edge] (a3)--(b1);
\draw[edge] (a2)--(c4);
\draw[edge] (a0)--++(0,5mm);
\draw[edge] (b0)--++(0,5mm);
\draw[edge] (c0)--++(0,5mm);
\draw[edge] (a1)--++(-4mm,4mm);
\draw[edge] (c1)--++(-4mm,4mm);
\draw[edge] (b2)--++(-4mm,-4mm);
\draw[edge] (c2)--++(-4mm,-4mm);
\draw[edge] (b3)--++(4mm,-4mm);
\draw[edge] (c3)--++(4mm,-4mm);
\node[below=1mm] at (a0) {$v_0$};
\node[below right] at (a4) {$v_1$};
\node[right] at (a3) {$v_2$};
\node[left] at (b1) {$v'_2$};
\node[right] at (b4) {$u_2=v'_1$};
\node[right] at (c4) {$v'_3$};
\draw[mediums] (a4) to[out=35,in=90] (b4);
    \path (a1)--(a4) node[below,midway] {$C$};
    \path (b1)--(b4) node[below,midway] (L) {$C_{v'_2}$};
\draw (L) node[xshift=-3.5cm] {$C_{v'_3}$};
\end{tikzgraph}
        \caption{If~$C$ and~$C_{v'_2}$ are $S$-adjacent cycles of length~$5$
        with~$S$-degree~$1$, then~$C_{v'_3}$ cannot be a cycle of length~$5$
        with~$S$-degree~$0$, for otherwise~$(S\setminus\{v_1v'_1\})\cup\{v_2v'_2,v_3v'_3\}$ would contradict the maximality
        of~$S$. Consequently,~$C_{v'_3}$, does not send charge to~$C$
        by~(R2).}\label{fig-add1}
    \end{center}
\end{figure}

If~$C$ has length~$5$, then let us write~$C=v_0\dotso v_4$ with~$v_1v'_1\in S$.
Recall that~$v_1v'_1$ cannot be medium because the $S$-component to which~$C$
belongs is not an odd cycle.  Furthermore,~$C$ can receive charge only by~(R2).
Observe that none of~$C_{v'_0}$ and~$C_{v'_2}$ is an odd cycle with $S$-degree~$0$
for otherwise adding the edge~$v_0v'_0$ or~$v_2v'_2$ to~$S$ would contradict
its choice. Therefore~$C$ can receive charge only through~$v_3$ or~$v_4$, for a
total of at most~$2/5$. It follows that if~$C$ sends~$1/5$ through each
of~$v_0$ and~$v_2$, due to rules~(R3) and~(R4), then its final charge surely is
at most~$\frac45\cdot\abs{V(G)}$.  Let us identify precisely when~$C$ sends
charge through~$v_2$, the case for~$v_0$ being identical.  If~$C_{v'_2}$ is not
a cycle of length~$5$ with $S$-degree~$1$, then~(R3) applies. So assume that~$C_{v'_2}$ is
a cycle of length~$5$ with $S$-degree~$1$, written~$v'_2u_1u_2u_3u_4$. None of~$u_1$
and~$u_4$ is incident to an edge in~$S$, for otherwise adding the
edge~$v_2v'_2$ to~$S$ would contradict its maximality.  So we can assume
without loss of generality that $u_2u'_2\in S$.  Now we observe that
if~$C_{u'_2}\neq C$, then~$C_{u'_2}$ cannot be a cycle of $S$-degree~$1$, for
otherwise~$(S\setminus\{u_2u'_2\})\cup\{v_2v'_2\}$ would contradict the fact
that~$S$, among all edge-selections of maximum order, creates the maximum number of cycles with $S$-degree~$2$. Therefore in this
case~$C_{u'_2}$ is a cycle of $S$-degree~$2$ and hence~$C$ sends~$1/5$ to~$C_{u'_2}$
by~(R4). It follows that the only case where~$C$ does not send charge
through~$v_2$ is when~$C_{u'_2} = C$ and hence~$v'_1=u_2$
(or~$v'_1=u_3$).
In this situation, we argue that~$C$ cannot receive charge
through~$v_3$. Indeed~$C$ can receive charge through~$v_3$ only by~(R2), which
applies if and only if~$C_{v'_3}$ is a cycle of length~$5$ and $S$-degree~$0$, as illustrated in Figure~\ref{fig-add1}. In this
case,~$(S\setminus\{v_1v'_1\})\cup\{v_2v'_2,v_3v'_3\}$ would contradict
the maximality of~$S$. Consequently, we proved that either~$C$ sends~$1/5$
through~$v_2$ or~$C$ receives nothing through~$v_3$.  By symmetry of the roles
played by~$v_2$ and~$v_0$, either~$C$ sends~$1/5$ through~$v_0$ or~$C$ receives
nothing through~$v_4$.  Since~$C$ can receive charge only through~$v_3$
and~$v_4$, we therefore conclude that the final charge of~$C$ is not greater
than its charge before applying~(R2)--(R4), that is~$4$.

It remains to deal with cycles with $S$-degree~$2$.
Let us write~$C=v_0\dotso v_{2k}$, where~$k\ge2$, with~$v_1v'_1\in S$
and~$v_2v'_2\in S$.
Suppose first that $G_K$ is not an odd cycle, where~$K$ is the $S$-component
to which~$C$ belongs. In particular, none of~$v_1v'_1$ and~$v_2v'_2$ is medium.
In this case, we show that the charge of~$C$ is at most~$\frac45\cdot(2k+1)$, with
equality only if~$k=2$.
Observe that, for each~$i\in\{1,2\}$,
the cycle~$C$ can receive some charge through~$v_i$ only if~$C_{v'_i}$ is a
cycle of length~$5$ and $S$-degree~$1$.  Further, according to~(R4), the
cycle~$C$ can receive at most~$2/5$ through~$v_i$, because only the two
vertices at distance two from~$v'_i$ on~$C_{v'_i}$ can be involved in an
application of~(R4).  As a result, it is enough to prove that if~$C$
receives~$2/5$ through~$v_i$, then~$C$ does not receive any charge
through~$v_{6-3i}$. More explicitly, if~$C$ receives~$2/5$ through~$v_1$
then~$C$ receives nothing through~$v_3$; and if~$C$ receives~$2/5$
through~$v_2$ then~$C$ receives nothing through~$v_0$. In total, the final
charge of~$C$ would then be at most~$3+(2k+1)/5$, which is at
most~$\frac45\cdot(2k+1)$ since~$k\ge2$, with equality if and only if~$k=2$.

\begin{figure}[!th]
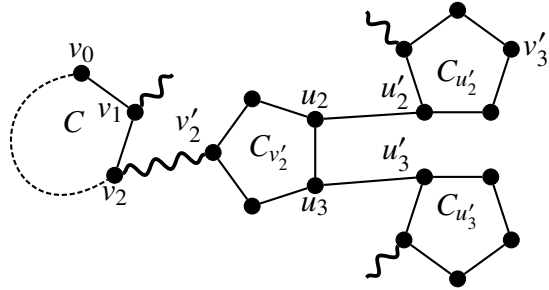
 
    \begin{center}
\begin{tikzgraph}
 \foreach \i in {0,3,4}{
    \draw (90+72*\i:.75cm) node[vertex] (a\i) {};
    \draw (180+72*\i:.75cm) node[vertex,xshift=2.5cm,yshift=-.3034cm] (b\i) {};
    \draw (90+72*\i:.75cm) node[vertex,xshift=5cm,yshift=.8385255cm] (c\i) {};
    \draw[rotate=180] (90+72*\i:.75cm) node[vertex,xshift=5cm,yshift=-1.235cm] (d\i) {};
}
 \foreach \i in {1,2}{
    \draw (180+72*\i:.75cm) node[vertex,xshift=2.5cm,yshift=-.3034cm] (b\i) {};
    \draw (90+72*\i:.75cm) node[vertex,xshift=5cm,yshift=.8385255cm] (c\i) {};
    \draw[rotate=180] (90+72*\i:.75cm) node[vertex,xshift=5cm,yshift=-1.235cm] (d\i) {};
}

\foreach \i\j in {3/4,4/0} {
  \draw[edge] (a\i)--(a\j);
  \draw[edge] (b\i)--(b\j);
  \draw[edge] (c\i)--(c\j);
  \draw[edge] (d\i)--(d\j);
}
\foreach \i\j in {0/1,1/2,2/3} {
  \draw[edge] (b\i)--(b\j);
  \draw[edge] (c\i)--(c\j);
  \draw[edge] (d\i)--(d\j);
}

\draw[edge,dashed,dash pattern=on 2pt off 1pt] (a0)  ..controls (-1.5,.5) and (-1,-1.5).. (a3);
\draw[edge] (b3)--(c2);
\draw[edge] (b2)--(d3);
\draw (c2) node[above left=-1mm and .5mm] {$u'_2$};
\draw (d3) node[above left=-.7mm and .3mm] {$u'_3$};
\node[above] at (a0) {$v_0$};
\node[left] at (a4) {$v_1$};
\node[below] at (a3) {$v_2$};
\node[above left] at (b0) {$v'_2$};
\node[above] at (b3) {$u_2$};
\node[below] at (b2) {$u_3$};
\node[right] at (c4) {$v'_3$};
\draw[mediums] (a3)--(b0);
\draw[mediums] (c1)--++(-5mm,5mm);
\draw[mediums] (d4)--++(-5mm,-5mm);
\draw[mediums] (a4)--++(5mm,5mm);
    \node at (-1mm,1mm) {$C$};
    \draw (2.5cm,-.3034cm) node {$C_{v'_2}$};
    \path (c1)--(c4) node[below,midway] {$C_{u'_2}$};
    \path (d1)--(d4) node[above,midway] {$C_{u'_3}$};
\end{tikzgraph}
        \caption{If the cycle~$C$ receives~$2/5$ through~$v_2$ by~(R4) and~$C_{u'_2}\neq C_{u'_3}$
        then $(S\setminus\{v_2v'_2\})\cup\{u_2u'_2,u_3u'_3\}$ contradicts the maximality of the edge-selection~$S$.}\label{fig-add2}
    \end{center}
\end{figure}

Let us establish the assertion above: assume without loss of generality that~$C$
receives~$2/5$ through~$v_2$ because of~(R4). Writing~$C_{v'_2}=v'_2u_1u_2u_3u_4$, we deduce that each
of~$C_{u'_2}$ and~$C_{u'_3}$ is a cycle of length~$5$ and $S$-degree~$1$. In
addition, by the definition of~(R4) for each~$i\in\{2,3\}$ the vertex
of~$C_{u'_i}$ incident to an edge in~$S$ is a neighbour of~$u'_i$
on~$C_{u'_i}$.  Suppose first that~$C_{u'_2}\neq C_{u'_3}$, as illustrated in Figure~\ref{fig-add2}.
Then~$(S\setminus\{v_2v'_2\})\cup\{u_2u'_2,u_3u'_3\}$ contradicts the maximality of~$S$.
If, on the contrary,~$C_{u'_2}=C_{u'_3}$, then without loss of generality we may write~$C_{u'_2}=w_0u'_2w_1w_2u'_3$
with~$w_0w'_0\in S$, as illustrated in Figure~\ref{fig-add3}. Now, if~$C$ receives charge through~$v_0$ then~$C_{v'_0}$ is a cycle of length~$5$
and $S$-degree~$0$ or~$1$.
In the latter case, we notice that the vertex of~$C_{v'_0}$
incident to an edge in~$S$ is consecutive to~$v'_0$ on~$C_{v'_0}$ and, consequently in both cases~$C_{v'_0}\neq C_{u'_2}$.
From this and the fact that~$v'_0\notin\{u'_2,u'_3\}$, we deduce
that in any case~$(S\setminus\{v_2v'_2\})\cup\{v_0v'_0,u_2u'_2\}$ contradicts the choice
of~$S$.

\begin{figure}[!th]
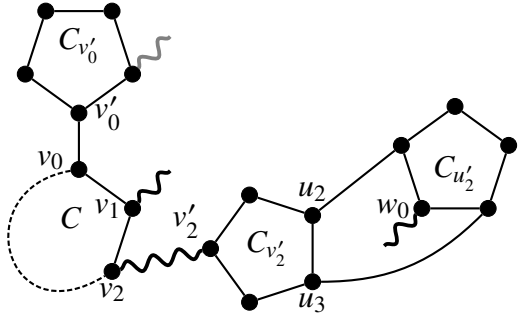
 
    \begin{center}
\begin{tikzgraph}
 \foreach \i in {0,3,4}{
    \draw (90+72*\i:.75cm) node[vertex] (a\i) {};
    \draw (180+72*\i:.75cm) node[vertex,xshift=2.5cm,yshift=-.3034cm] (b\i) {};
    \draw (90+72*\i:.75cm) node[vertex,xshift=5cm,yshift=.8385255cm] (c\i) {};
    \draw[rotate=180] (90+72*\i:.75cm) node[vertex,yshift=2.25cm] (d\i) {};
}
 \foreach \i in {1,2}{
    \draw (180+72*\i:.75cm) node[vertex,xshift=2.5cm,yshift=-.3034cm] (b\i) {};
    \draw (90+72*\i:.75cm) node[vertex,xshift=5cm,yshift=.8385255cm] (c\i) {};
    \draw[rotate=180] (90+72*\i:.75cm) node[vertex,yshift=2.25cm] (d\i) {};
}

\foreach \i\j in {3/4,4/0} {
  \draw[edge] (a\i)--(a\j);
  \draw[edge] (b\i)--(b\j);
  \draw[edge] (c\i)--(c\j);
  \draw[edge] (d\i)--(d\j);
}

\foreach \i\j in {0/1,1/2,2/3} {
  \draw[edge] (b\i)--(b\j);
  \draw[edge] (c\i)--(c\j);
  \draw[edge] (d\i)--(d\j);
}

\draw[edge,dashed,dash pattern=on 2pt off 1pt] (a0)  ..controls (-1.5,.5) and (-1,-1.5).. (a3);
\draw[edge] (b3)--(c1);
\draw[edge] (a0)--(d0);
\draw[edge] (b2) to[out=0, in=-135] (c3);
\node[above left=-1.5mm and .5mm] at (a0) {$v_0$};
\node[right=.5mm] at (d0) {$v'_0$};
\node[left] at (a4) {$v_1$};
\node[below] at (a3) {$v_2$};
\node[above left] at (b0) {$v'_2$};
\node[above=.5mm] at (b3) {$u_2$};
\node[below] at (b2) {$u_3$};
\node[left] at (c2) {$w_0$};
\draw[mediums] (a3)--(b0);
\draw[mediums] (c2)--++(-5mm,-5mm);
\draw[mediums,gray] (d1)--++(5mm,5mm);
\draw[mediums] (a4)--++(5mm,5mm);
    \node at (-1mm,1mm) {$C$};
    \draw (2.5cm,-.3034cm) node {$C_{v'_2}$};
    \path (c1)--(c4) node[below,midway] {$C_{u'_2}$};
    \path (d1)--(d4) node[above,midway] {$C_{v'_0}$};
\end{tikzgraph}
        \caption{If the cycle~$C$ receives~$2/5$ through~$v_2$ by~(R4) and~$C_{u'_2}=C_{u'_3}$, then~$C$ cannot
        receive~$1/5$ through~$v_0$ by~(R2) or~(R3), for otherwise~$(S\setminus\{v_2v'_2\})\cup\{v_0v'_0,u_2u'_2\}$
        would contradict the maximality of the edge-selection~$S$. (The gray waved edge belongs to~$S$ only if~$C_{v'_0}$
        has $S$-degree~$1$.)}\label{fig-add3}
    \end{center}
\end{figure}

It remains to deal with the case where~$C$ belongs to an $S$-component~$K$
such that~$G_K$ is an odd cycle. In this case, we prove the sum of the charges
of all cycles in~$K$ to be less than~$\frac45$ times the number of vertices belonging to
cycles in~$K$.
Every cycle in~$K$ is incident to exactly two edges in~$S$. It follows
that none of these cycles receive charge by~(R4).
Moreover, exactly one edge in~$S$ associated with~$K$ is medium. Consequently,
if~$C\in K$ is not incident to a medium edge that belongs to~$S$, then the final
charge of~$C$ is at most~$3+\frac15\cdot(\abs{V(C)}-2)$. If~$C$ is one of the two cycles
incident to the medium edge in~$S$ associated with~$K$, then the final charge of~$C$
is at most~$\frac72+\frac15\cdot(\abs{V(C)}-2)$. It now suffices to sum these quantities
over the cycles in~$K$: let us write~$K=\{C_1,\dotsc,C_t\}$ where~$t$ is an odd number
at least~$3$. Setting~$\ell_i\coloneqq\abs{V(C_i)}$ for
each~$i\in\{1,\dotsc,t\}$, the sum of the final charges of the cycles in~$K$ is
at most
\[
    3t+1+\sum_{i=1}^{t}\frac{\ell_i-2}{5} = 1+\frac{1}{5}\left(13t+\sum_{i=1}^{t}\ell_i\right).
\]
It only remains to show that this last quantity is less than~$\frac45\cdot\sum_{i=1}^{t}\ell_i$,
that is,
\begin{equation}\label{eq-last}
5+13t < 3\cdot\sum_{i=1}^{t}\ell_i.
\end{equation}
Since each cycle in~$F$ has length at least~$5$, one has~$\sum_{i=1}^t\ell_i\ge 5t$,
and hence~\eqref{eq-last} holds because~$t\ge3$.

Looking at the above inequalities, we observe that as soon as~$F$ contains a cycle~$C$
of length different from~$5$, then the number of medium edges is less
than~$\frac45\cdot\abs{V(G)}$.
Therefore, the number of medium edges obtained is strictly
less than~$\frac45\cdot\abs{V(G)}$ unless~$F$ contains only cycles of length~$5$.
As it turns out, it has been proved~\cite{DMP08} that every connected bridgeless
cubic graph different from the Petersen graph admits a $2$-factor containing a cycle
of length different from~$5$.\footnote{See also~\cite{DK} for a different and
short argument.} This concludes the proof of Theorem~\ref{thm-main}.

\section{Further work}
We point out that, using more involved discharging rules and a lengthier
analysis, one can show that there exists a positive~$\varepsilon$ (which we did not
try to optimise) such that for every connected bridgeless cubic graph~$G$
different from the Petersen graph, there exists a $4$-edge-colouring yielding
at most~$(4/5-\varepsilon)\abs{V(G)}$ medium edges. 

It seems stimulating to try and obtain upper bounds for the least possible number
of medium edges in a $k$-edge-colouring of a bridgeless cubic graph. As we saw,
this number is~$0$ if~$k\ge7$ and at most~$\frac45\cdot\abs{V(G)}$ if~$k=4$.
Since the Petersen colouring conjecture states that this number should be~$0$
when~$k=5$, can one obtain at least a sub-linear (in the number of vertices)
upper bound in this case? What can be proved when~$k=6$?

\end{document}